\documentclass[11pt]{article}

\usepackage[colorlinks=true,citecolor=black,linkcolor=black,urlcolor=blue]{hyperref}
\usepackage[english]{babel}
\usepackage[utf8]{inputenc}
\usepackage{latexsym}
\usepackage{booktabs}
\usepackage{amsfonts}
\usepackage[pdftex]{graphicx}
\usepackage{subfig}
\usepackage{caption}
\usepackage{textcomp}
\usepackage{amsmath,amssymb,amsthm}
\usepackage{tikz,color,graphicx}
\usepackage{pgfplots}
\usepackage{floatflt,epsfig}
\usepackage{bbm}
\usepackage{blkarray}
\usepackage{mathtools}
\usepackage{tikz}
\usetikzlibrary{positioning,chains,fit,shapes,calc}
\usetikzlibrary{arrows,intersections}
\usepackage{mathabx}

\theoremstyle{plain}
\newtheorem{theorem}{Theorem}[section]
\theoremstyle{definition}

\theoremstyle{definition}

\theoremstyle{definition}

\theoremstyle{definition}

\theoremstyle{plain}

\theoremstyle{plain}

\theoremstyle{plain}
\newtheorem{lemma}[theorem]{Lemma}
\theoremstyle{definition}

\theoremstyle{definition}

\theoremstyle{plain}
\newtheorem{conjecture}[theorem]{Conjecture}

\numberwithin{equation}{section}
\numberwithin{theorem}{section}

\def\dfrac#1#2{\lower0.15ex\hbox{\large$\frac{#1}{#2}$}}
\def\Z{\mathbb{Z}}
\def\R{\mathbb{R}}
\def\P{\mathbb{P}}
\def\E{\mathbb{E}}
\def\Var{\mathrm{Var}}
\def\Cov{\mathrm{Cov}}

\begin{document}

\pagestyle{plain}
\pagenumbering{arabic}


\title{From stability to chaos in last-passage percolation}
\author{Daniel Ahlberg\thanks{Department of Mathematics, Stockholm University, Sweden.\newline\hspace*{0.5cm}{\tt \{daniel.ahlberg\}\{mia\}\{matteo.sfragara\}@math.su.se.}}\and Maria Deijfen\footnotemark[1] \and Matteo Sfragara\footnotemark[1]}
\date{September 2023}
\maketitle

\begin{abstract}

We study the transition from stability to chaos in a dynamic last passage percolation model on $\mathbb{Z}^d$ with random weights at the vertices. Given an initial weight configuration at time $0$, we perturb the model over time in such a way that the weight configuration at time $t$ is obtained by resampling each weight independently with probability $t$. On the cube $[0,n]^d$, we study geodesics, that is, weight-maximizing up-right paths from $(0,0, \dots, 0)$ to $(n,n, \dots, n)$, and their passage time $T$. Under mild conditions on the weight distribution, we prove a phase transition between stability and chaos at $t \asymp \frac{1}{n}\mathrm{Var}(T)$. Indeed, as $n$ grows large, for small values of $t$, the passage times at time $0$ and time $t$ are highly correlated, while for large values of $t$, the geodesics become almost disjoint.

\vspace{0.3cm}

\noindent \emph{Keywords:} Last-passage percolation, geodesics, stability, chaos.

\vspace{0.2cm}

\noindent AMS 2010 Subject Classification: 60K35.
\end{abstract}


\section{Introduction}
\label{sec:intro}

\subsection{Motivation and background}

In statistical mechanics, the energy landscapes of many disordered systems have a complex geometry, where the configuration with the lowest energy -- the ground state -- corresponds to the bottom of the deepest valley of the landscape. The phenomenon of `chaos' is characterized by energy landscapes with many valleys, and many roughly orthogonal near-ground states, resulting in a system where a slight perturbation of the disorder will lead to a significant change of the ground state.

The mathematical study of chaos in disordered systems was initiated by Chatterjee in two preprints~\cite{C08,C09}, which were later combined into a book~\cite{C14}. Chatterjee established a precise relation between fluctuations of the ground state energy and the effect on the ground state of a perturbation of the medium. This relation allowed him to deduce an equivalence between `superconcentration', that is, sub-Gaussian fluctuations, and chaos for certain Gaussian disordered systems. By establishing superconcentration for a Gaussian directed polymer, and the top eigenvalue of a Gaussian matrix, he obtained the first evidence of a chaotic behaviour. More recently, the precise location of the transition from stability to chaos has been established for the top eigenvector of Wigner matrices by Bordenave-Lugosi-Zhivotovskiy~\cite{BLZ20} and in the context of Brownian last-passage percolation by Ganguly-Hammond~\cite{GH20a,GH20b}.

The analysis in~\cite{C14} and~\cite{GH20a,GH20b} is specific to the Gaussian context, and based on spectral techniques. While the methods of~\cite{BLZ20} are not specific to the Gaussian setting, they are applied to eigenvalues and the corresponding vectors of random matrices, which is a well-understood topic. In a companion paper~\cite{ADS22}, our goal has been to show that the equivalence between superconcentration and chaos is part of a more general principle, achieved by establishing the connection in the context of first-passage percolation. In the current paper we proceed in this vein, and determine the precise location of the transition from stability to chaos in non-Gaussian and non-integrable models of last-passage percolation.

\subsection{Model and results}

Last-passage percolation is a model for spatial growth, defined on the integer lattice $\Z^d$. For integers $n\ge1$, consider the cube $V = [0,n]^d \cap \mathbb{Z}^d$ and let $\omega = (\omega_v)_{v \in V}$ be a collection of i.i.d.\ weights drawn from some probability distribution $F$ on $[0, \infty)$. Let $T = T(\omega)$ denote the maximal weight-sum picked up along any directed (up-right) path from $(0,0, \dots, 0)$ to $(n,n, \dots,n)$, i.e.,
\begin{equation}
T = T(\omega) = \max_{\gamma \in \Gamma} \sum_{v \in \gamma} \omega_v,
\end{equation}
where $\Gamma$ is the set of all nearest-neighbour paths from $(0,0, \dots, 0)$ to $(n,n, \dots, n)$ whose steps follows either of the $d$ coordinate vectors $(1,0, \dots, 0), (0,1,0, \dots, 0), \dots, (0, \dots, 0, 1)$. We will refer to $T$ as the {\bf passage time} between $(0,0, \dots, 0)$ and $(n,n,\dots,n)$ due to its interpretation as the occupancy time in the related corner growth model. We refer to any weight-maximizing path, i.e.\ any path $\pi\in\Gamma$ that attains the maximum in $T$, as a {\bf geodesic} between the same points. When $F$ is continuous there is a unique such path.

It follows from \cite[Theorem 2.3]{M02} that $\E[T]$ grows at most linearly in $n$ provided that $\int_0^\infty (1-F(x))^{1/d}dx<\infty$. In our arguments, we will need that the passage time grows at most linearly for a configuration consisting of squared weights. We will hence throughout assume that
\begin{equation}\label{Fass}
\int_0^\infty (1-F(\sqrt{x}))^{1/d}dx<\infty,
\end{equation}
which is slightly stronger than finite $2d$-moment.

In 1986, the work of Kardar-Parisi-Zhang~\cite{KPZ} gave predictions for the asymptotic behaviour of a large class of \textit{planar} spatial growth models. Via the analysis of a differential equation, they were led to predict that the fluctuations of $T$ around its mean are of the order $n^{1/3}$ and that the fluctuations of the/any geodesic associated with $T$ are of the order $n^{2/3}$. Remarkable work of~\cite{J00a}, inspired by \cite{BDJ}, verified these predictions in last-passage percolation with exponential or geometric weight distribution, which have come to be referred to as the `integrable' or `exactly solvable' setting. Moreover, their work also determines that $T$, when centered and appropriately normalized, converges to the Tracy-Widom distribution (also known to be the asymptotic distribution of the largest eigenvalue of a GUE random matrix). In particular, it follows from~\cite{BDMLMZ01} that when $d=2$ and $F$ is the exponential or geometric distribution, then
\begin{equation}
\Var(T)=\Theta(n^{2/3}).
\end{equation}
This result will be relevant in combination with our main result below.

In this paper we will consider a dynamic version of last-passage percolation, in order to address how it is affected when exposed to perturbations of the weight configuration. Recall that $V = [0,n]^d \cap\Z^d$, where $n\ge1$ is an integer, and let $\omega = (\omega_v)_{ v \in V}$ and $\omega' =(\omega'_v)_{v \in V}$ be independent weight configurations, that is, collections of independent variables distributed as $F$. Let $U = (U_v)_{v\in V}$ be a collection of independent random variables uniformly distributed on $[0,1]$, and independent of $(\omega, \omega')$. For each $t \in [0,1]$ we obtain a weight configuration $\omega(t)$ according to
\begin{equation}
\omega_v(t) := \begin{cases}
\omega_v, & \text{if } U_v > t, \\
\omega_v', & \text{if } U_v \leq t.
\end{cases}
\end{equation}
We will think of $t\in[0,1]$ as `time' (not to be confused with the passage time $T$) and of $(\omega(t))_{t\in[0,1]}$ as a weight configuration evolving dynamically over time. For $t>0$ the configuration $\omega(t)$ corresponds to a perturbation of $\omega(0)$, and $t$ dictates the magnitude of the perturbation. (The coordinate-wise correlation of the two configurations at time $t$ equals $1-t$.) For the purposes of this paper, an alternative construction would be to update weights according to independent Poisson clocks; our results below can be recast in this language via a re-parametrization of time.

For $t \in [0,1]$, we denote by $T_t$ the passage time with respect to the configuration $\omega(t)$, and by $\pi_t$ the set of vertices contained in some geodesic for $T_t$. Recall that if $F$ is continuous then $\pi_t$ is the unique maximizer of $T_t$, while if $F$ has atoms there may be multiple geodesics with the same passage time and $\pi_t$ denotes their union. Our main result addresses the transition from stability to chaos in the context of last-passage percolation. In the planar and integrable setting (when $F$ is exponential or geometric) our main result states that this transition occurs at $t\asymp n^{-1/3}$ in the following sense:
\begin{description}
\item[Stability:] For $t\ll n^{-1/3}$ we have $\mathrm{Corr}(T_0,T_t)=1-o(1)$.
\item[Chaos:] For $t\gg n^{-1/3}$ we have $\E[|\pi_0\cap\pi_t|]=o(1)$.
\end{description}

In fact, our methods show that the analogous result holds in a more general context, and not only in the exactly solvable setting. We shall first require that $F$ satisfies \eqref{Fass}. Second, we shall make the assumption that weights conditioned on being `large' have variance uniformly bounded from below, that is,
\begin{equation}
\label{assumption2}
\exists \, c > 0 \text{ such that } \mathrm{Var}(\omega_v \, | \, \omega_v >k) > c \text{ for all } k\ge0 \text{ and } v \in V.
\end{equation}
We show in the appendix that this assumption is, for instance, met by weight distributions $F$ with $1-F(x)=Cx^{-\gamma}$ for $\gamma>2$, or with $1-F(x)=C\exp(-\bar{\beta}x^\beta)$ for $\beta\in(0,1]$.

Our main result states that, under the above assumptions on the weight distribution, the transition from stability to chaos occurs at $t \asymp \frac{1}{n}\mathrm{Var}(T)$.

\begin{theorem}[From stability to chaos]
\label{thm:stabilityandchaos}
Consider last-passage percolation on $\mathbb{Z}^d$, for $d \geq 2$, with a weight distribution satisfying \eqref{Fass}. There exists a constant $C<\infty$ such that for all $n\ge1$ and $0<\alpha<\frac{n}{\Var(T)}$ the following two statements hold.
\begin{itemize}
\item[(i)] {\bf Stability:} For $t \leq \alpha\frac{1}{n}\mathrm{Var}(T)$, we have
\begin{equation}
\label{stability}
\mathrm{Corr}(T_0, T_t) \geq  1-C\alpha.
\end{equation}
\item[(ii)]
{\bf Chaos:} If, in addition, \eqref{assumption2} holds and $t \geq \alpha\frac{1}{n}\mathrm{Var}(T)$, then
\begin{equation}
\label{chaos}
\mathbb{E}[|\pi_0 \cap \pi_t|] \leq C\frac{n}{\alpha}.
\end{equation}
\end{itemize}
\end{theorem}

In the planar exactly solvable setting, we have that $\Var(T)=\Theta(n^{2/3})$ and the transition from stability to chaos hence occurs at $t=\Theta(n^{-1/3})$. The same asymptotic behaviour is predicted to prevail under mild conditions on the weight distribution. In higher dimensions fluctuations are predicted to be smaller still. A sub-linear $n/\log n$-upper bound on $\Var(T)$ was obtained in~\cite{C14,G12} for a large class of weight distributions referred to as `nearly Gamma'. This sub-linear upper bound is sufficient to establish that the transition from stability to chaos occurs at $t=o(1)$. However, for many such distributions, condition~\eqref{assumption2} will not hold. It would be interesting to extend Theorem~\ref{thm:stabilityandchaos} so that condition~\eqref{assumption2} is not required.

Theorem~\ref{thm:stabilityandchaos} is more precise than the corresponding results obtained for first-passage percolation in the companion paper~\cite{ADS22}. In~\cite{ADS22} we establish chaos in first-passage percolation for fixed $t>0$ for a large class of weight distributions (including continuous distribution with finite $2+\log$ moments). Although we expect that the transition from stability to chaos occurs at $t\asymp \frac{1}{n}\Var(T)$ also there, we have not been able to prove that. The reason we succeed in last-passage percolation, where we fail in first-passage percolation, is because the former is a maximization problem and the latter a minimization problem. While in first-passage percolation the weight distribution is bounded from below, in last-passage percolation it is (possibly) unbounded, and condition (1.5) implies that even if a vertex is on the geodesic because it is very heavy, it will still contribute to the influence.

\subsection{Discussion}

Let us emphasize that Theorem~\ref{thm:stabilityandchaos} establishes a transition from stability to chaos in a weaker sense than in~\cite{BLZ20,GH20a}, albeit by more flexible methods. Our theorem considers different quantities in the stable and chaotic regimes (the distance function and the distance-maximizing path, respectively), whereas the authors in~\cite{BLZ20,GH20a} consider stability and chaos of the same object (the top eigenvector and distance-maximizing path, respectively). In~\cite{BLZ20} this is possible due to the detailed understanding of eigenvectors of random matrices, and in~\cite{GH20a} due to the precise understanding of the geometry of near-ground states, obtained by the same authors in~\cite{GH20b}.

Our proof of Theorem~\ref{thm:stabilityandchaos} relies on the other hand on a covariance formula derived in~\cite{ADS22}, and requires no precise model-specific estimates, which we emphasize by working outside of the exactly solvable setting. The covariance formula alluded to concerns the function $Q_t:=\E[T_0T_t]$ defined for $t\in[0,1]$. Since the configurations at time 0 and at time 1 are independent, we may express the variance of $T$ as
\begin{equation}\label{variance}
\mathrm{Var}(T) = \mathbb{E}[T_0^2] - \mathbb{E}[T_0 T_1] = Q_0- Q_1 = - \int_0^1 \frac{d}{dt}Q_t \, dt.
\end{equation}
The core of the argument will be to relate the contribution to the derivative of $Q_t$ that comes from a vertex $v$ to the probability that $v$ belongs to both $\pi_0$ and $\pi_t$.

Looking beyond Theorem~\ref{thm:stabilityandchaos}, we expect that in the stable regime also the expected overlap between the two geodesics remains significant. Our belief is supported by the analogous behaviour established for Brownian last-passage percolation in~\cite{GH20a}.
\begin{conjecture}
\label{conjecture2}
If $t \ll \frac{1}{n}\mathrm{Var}(T)$, then $\mathbb{E}[|\pi_0 \cap \pi_t|] = \Theta(n)$ as $ n \to \infty$.
\end{conjecture}
We further expect that in the chaotic regime also the passage times decorrelate. While we are not aware of results of this kind for related models, heuristic reasoning involving the KPZ scaling relations suggests this should be the case.
\begin{conjecture}
\label{conjecture1}
If $t \gg \frac{1}{n}\mathrm{Var}(T)$, then $\mathrm{Corr}(T_0, T_t) = o(1)$ as $n \to \infty$.
\end{conjecture}


The study of decorrelation as an effect of small perturbations was initiated by Benjamini-Kalai-Schramm~\cite{BKS99} in the context of Boolean functions, and referred to as noise sensitivity. A substantial framework for the study of noise sensitivity has since been developed, but has remained largely restricted to the Boolean setting. The notion of chaos is strongly related to the notion of noise sensitivity introduced in \cite{BKS99}, according to which an event is noise sensitive if, with high probability, an arbitrary small random perturbation of the configuration gives almost no prediction of whether the event occurs. In general, while chaos refers to the energy-minimizing object (in our case the geodesics), noise sensitivity refers to the decorrelation of the energy of that object over time (the passage time). In this regard, we conjecture that the location for the transition from stability to chaos is also the right location for the transition from stability to noise sensitivity. In other words, for values of $t$ smaller than the threshold, the overlap between geodesics is still of order $n$, while for values of $t$ larger than the threshold, the passage times become decorrelated, hence they are noise sensitive. We expect both conjectures to be challenging to prove, possibly requiring different techniques.\medskip

\noindent
{\bf Outline of the paper.} The remainder of the paper is organized as follows. In Section~2 we describe a covariance formula from our companion paper \cite{ADS22}. In Section~3 we derive lower and upper bounds on the influence of a vertex in terms of the probability that the vertex belongs to the geodesic both at time 0 and at time $t$. These bounds are then combined with the covariance formula in the proof of Theorem~\ref{thm:stabilityandchaos} in Section~4.

\section{The covariance formula}
\label{sec:covariance}

The proof of Theorem~\ref{thm:stabilityandchaos} relies on a covariance formula derived in~\cite{ADS22}, which we describe next.

For $v \in V$ and $x \in [0, \infty)$, let $\sigma_v^{x}: [0, \infty)^{V} \to [0, \infty)^{V}$ be the operator that replaces the weight $\omega_v$ at $v$ by $x$. Write $T^{v \to x} := T \circ \sigma_v^x$ and let
\begin{equation}
\label{defDvx}
D_v^x T := T^{v \to x} - \int T^{v \to y}  \, dF(y).
\end{equation}
That is, $D_v^xT$ compares the travel time when the weight at $v$ is fixed to $x$ and when averaged over all possible values. We define the {\bf co-influence} of a vertex $v \in V$ at times $0$ and $t$ as
\begin{equation}
\label{definfluence}
\textrm{Inf}_v(t) := \int \E\big[D_v^x T_0 D_v^x T_t\big]\, dF(x).
\end{equation}

A standard coupling argument (see~\cite[Lemma 5]{ADS22}, but we believe that similar constructions have been used previously by other authors) shows that, for any function $f:[0,\infty)^V\to\R$, the outcomes $f(\omega(0))$ and $f(\omega(t))$ are positively correlated, that is,
\begin{equation}\label{eq:time_corr}
\E\big[f(\omega(0))f(\omega(t))\big]-\E[f]^2\ge0,
\end{equation}
and the correlation is non-increasing in $t$.
In particular, as a consequence, the co-influences are non-negative and non-increasing in $t$. For the same reason, the probability $\P(v\in\pi_0\cap\pi_t)$ is non-increasing as a function of $t$.

The co-influences was in~\cite{ADS22} related to $Q_t$ through the formula
\begin{equation}
\label{margulisrusso}
-\frac{d}{dt}Q_t = \sum_{v \in V} \textrm{Inf}_v(t).
\end{equation}
This quantity is thus non-negative, so that $Q_t$ is non-increasing in $t$. Combining~\eqref{variance} and~\eqref{margulisrusso} gives the formula
\begin{equation}
\label{varianceinfluence}
\mathrm{Var}(T) = \int_0^1 \sum_{v \in V} \mathrm{Inf}_v(t) \, dt.
\end{equation}
Since co-influences are non-increasing, the sum of influences $\sum_{v \in V} \mathrm{Inf}_v(0)$ gives an upper bound on the variance, reminiscent of an Efron-Stein or Poincar\'e inequality.

\section{Bounding the influence}
\label{sec:boundingtheinfluence}

The proof of the main result of the paper will go via the covariance formula~\eqref{varianceinfluence}. The key step in the proof will be to show that the influence of a vertex $v$ is proportional to the probability that it is part of the geodesic. That is the goal of this section, which will start with some additional notation.

For $v\in V$ and $t \in [0,1]$ fixed, we let $k_v(t)$ denote the smallest non-negative value that the weight at $v$ can take on for $v$ to be on some geodesic for $T_t$. Since $T_t$ is the maximal weight sum on directed paths from $(0,0, \dots, 0)$ to $(n,n, \dots, n)$, it follows that $k_v(t)$ is almost surely finite for all
$v$. By definition, the weight $k_v(t)$ is determined by $(\omega_u(t))_{u\neq v}$ and
\begin{equation}
\label{defk*}
\{v \in \pi_t\} = \{\omega_v(t) \geq k_v(t)\}.
\end{equation}
Moreover, the vertex $v$ is on all geodesics of $T_t$ if $\omega_v(t)>k_v(t)$.

Let $\mathcal{F}_v$ be the $\sigma$-algebra generated by the weights at vertices other than $v$, that is,
$$
\mathcal{F}_v = \sigma\big(\big\{\omega_u(t): u \in V, u \neq v, t \in [0,1]\big\}\big),
$$
and note that both $k_v(0)$ and $k_v(t)$ are determined by $\mathcal{F}_v$. In particular, $\omega_v(t)$ and $k_v(t)$ are independent, so if $\tilde\omega$ denotes a generic random variable with distribution $F$, and independent of everything else, then we have from~\eqref{defk*} that
\begin{equation}
\label{eqprobevents}
\P(v \in \pi_t \, | \, \mathcal{F}_v) = \P(\omega_v(t) \geq k_v(t) \, | \, \mathcal{F}_v) = \P(\tilde\omega \geq k_v(t) \, | \, \mathcal{F}_v).
\end{equation}
By similar reasoning, we also have that
\begin{equation}\label{eqprobevents2}
\P(v \in \pi_0\cap\pi_t \, | \, \mathcal{F}_v) =
\P(\omega_v(0)\geq k_v(0),\omega_v(t)\geq k_v(t)\, | \, \mathcal{F}_v)\leq
\P(\tilde\omega \geq \max\{k_v(0),k_v(t)\} \, | \, \mathcal{F}_v).
\end{equation}

Also note that, by the definition of $k_v(t)$ as the smallest value for the weight at $v$ that causes the geodesic to go through $v$, we have for $x\ge0$ that
$$
T_t^{v \to x} = T_t^{v \to k_v(t)} + (x-k_v(t))_+,
$$
where $(\,\cdot\,)_+$ denotes the positive part of the expression within brackets.
Hence
\begin{equation}
\label{eq1}
D_v^x T_t = T_t^{v \to x} - \int T_t^{v\to y}\,dF(y) =  (x-k_v(t))_+ - \int (y-k_v(t))_+ \, dF(y).
\end{equation}
This will be the basis for the following characterization of the co-influence.

\begin{lemma}
\label{lemma:infcov}
Suppose that $F$ satisfies \eqref{Fass}, and let $\tilde\omega$ denote a generic random variable distributed as $F$.
Then the co-influence of $v$ at time $t$ can be written as
\begin{equation}
\label{influencecovariance}
\mathrm{Inf}_v(t) = \E \Big[ \mathrm{Cov}\big(  (\tilde\omega-k_v(0))_+,(\tilde\omega-k_v(t))_+ \, \big| \, \mathcal{F}_v  \big)\Big].
\end{equation}
\end{lemma}

\begin{proof}
Since $\tilde\omega$ is $F$-distributed, we have that
$$
\int (y-k_v(t))_+ \, dF(y)=\E\big[(\tilde\omega-k_v(t))_+\big|\mathcal{F}_v\big].
$$
Together with~\eqref{eq1}, this shows that
$$
\int D_v^xT_0 D_v^xT_t\,dF(x) = \mathrm{Cov}\big(  (\tilde\omega-k_v(0))_+,(\tilde\omega-k_v(t))_+ \, \big| \, \mathcal{F}_v \big).
$$
The result follows by taking expectation.
\end{proof}

Next we derive an upper bound on the influence of $v$ for $t=0$ in terms of the geodesic.

\begin{lemma}[Upper bound]
\label{lemmaUB}
Suppose that $F$ satisfies \eqref{Fass}. Then
\begin{equation}
\label{UBinf}
\mathrm{Inf}_v(0) \leq \E \left[ \omega_v(0)^2 \, \mathbbm{1}_{\{ v \in \pi_0 \}} \right].
\end{equation}
\end{lemma}

\begin{proof}
When $t=0$, from Lemma~\ref{lemma:infcov}, we have that
$$
\mathrm{Inf}_v(0) = \E \big[ \mathrm{Var}\big(  (\tilde\omega-k_v(0))_+\, \big| \, \mathcal{F}_v \big)\big] \leq \E \big[\E \big[  (\tilde\omega-k_v(0))^2 \mathbbm{1}_{\{ \tilde\omega \geq k_v(0)\}} \, \big| \, \mathcal{F}_v \big] \big].
$$
Consequently, by independence of $\omega_v(0)$ and $k_v(0)$, we obtain from~\eqref{defk*} that
$$
\mathrm{Inf}_v(0) \leq \E \big[  (\omega_v(0)-k_v(0))^2 \mathbbm{1}_{\{ \omega_v(0) \geq k_v(0)\}} \big] \leq \E \big[ \omega_v(0)^2 \, \mathbbm{1}_{\{ v \in \pi_0 \}} \big],
$$
as required.
\end{proof}

\begin{lemma}[Lower bound]
\label{lemmaLB}
Suppose that $F$ satisfies \eqref{Fass} and that~\eqref{assumption2} holds. Then there exists $c>0$ such that for all $v\in V$ and $t\in[0,1]$ we have
\begin{equation}
\label{LBinf}
\mathrm{Inf}_v(t) \geq c \, \P(v \in \pi_0 \cap \pi_t).
\end{equation}
\end{lemma}

\begin{proof}
Let $\tilde\omega$ be a generic $F$-distributed random variable independent of everything else. Then by Lemma~\ref{lemma:infcov} we have
\begin{equation}\label{eq:re-coinf}
\mathrm{Inf}_v(t)=\E\Big[\mathrm{Cov}\big(  (\tilde\omega-k_v(0))_+,(\tilde\omega-k_v(t))_+ \,\big| \, \mathcal{F}_v  \big)\Big].
\end{equation}
Let $A_0 = \{\tilde\omega \geq k_v(0)\}$,  $A_t = \{\tilde\omega \geq k_v(t)\}$ and set $A = A_0 \cap A_t = \{ \tilde\omega \geq \max \{k_v(0),k_v(t)\} \}$. Then $(\tilde\omega-k_v(t))_+=(\tilde\omega-k_v(t))\mathbbm{1}_{A_t}$, and the conditional covariance in the right-hand side of~\eqref{eq:re-coinf} can be rewritten as
\begin{equation}\label{eq:cov_cond}
\begin{split}
&\, \E\big[ (\tilde\omega-k_v(0))(\tilde\omega-k_v(t)) \mathbbm{1}_{A} \, \big| \, \mathcal{F}_v \big] \\
& - \E\big[ (\tilde\omega-k_v(0)) \mathbbm{1}_{A_0} \, \big| \, \mathcal{F}_v\big]  \E\big[ (\tilde\omega-k_v(t))\mathbbm{1}_{A_t} \, \big| \, \mathcal{F}_v\big] \\
= & \,  \E\big[ (\tilde\omega-k_v(0))(\tilde\omega-k_v(t)) \, \big| \, A, \mathcal{F}_v \big] \P (A \, | \, \mathcal{F}_v) \\
& - \E\big[ \tilde\omega-k_v(0)  \, \big| \, A_0, \mathcal{F}_v\big]  \P(A_0 \, | \, \mathcal{F}_v)\, \E\big[ \tilde\omega-k_v(t) \, \big| \, A_t, \mathcal{F}_v\big] \P( A_t \, | \, \mathcal{F}_v).
\end{split}
\end{equation}
We claim that
\begin{equation}\label{eq:exp_ineq}
\E\big[ \tilde\omega-k_v(t)  \, \big| \, A_t, \mathcal{F}_v\big]
\le \E\big[ \tilde\omega-k_v(t)  \, \big| \, A, \mathcal{F}_v\big].
\end{equation}
On the event that $k_v(0)\le k_v(t)$, this is a trivial statement. On the event that $k_v(0)>k_v(t)$, the claim follows immediately from the two observations:
\begin{itemize}
\item[(i)] given a random variable $X$, if $x > k$, then
$$
\P(X > x \, | \, X > k) = \frac{\P(X >x)}{\P(X>k)} \geq \P(X >x);
$$
\item[(ii)] given two probability distributions $F$ and $\widetilde{F}$, if $F(x) \geq \widetilde{F}(x)$ for all $x\in\R$ and $g:\R\to\R$ is an increasing function, then
$$
\int g \, dF \leq \int g \, d \widetilde{F}.
$$
\end{itemize}

Next, we note that if $k_v(0)\le k_v(t)$, then $A_t\subseteq A_0$ which implies that $A=A_0\cap A_t=A_t$. If, on the other hand, $k_v(0)> k_v(t)$, then $A_0\subseteq A_t$ and hence $A=A_0\cap A_t=A_0$. In either case, since $k_v(0)$ and $k_v(t)$ are $\mathcal{F}_v$-measurable, we obtain that
\begin{equation}\label{eq:prob_ineq}
\P(A_0|\mathcal{F}_v)\P(A_t|\mathcal{F}_v)\le\P(A|\mathcal{F}_v).
\end{equation}
Combining~\eqref{eq:exp_ineq} and~\eqref{eq:prob_ineq} we obtain the following lower bound on~\eqref{eq:cov_cond}
\begin{equation}
\label{covcov}
\begin{split}
&\E\big[ (\tilde\omega-k_v(0))(\tilde\omega-k_v(t)) \, \big| \, A, \mathcal{F}_v \big] \P (A \, | \, \mathcal{F}_v) \\
&\quad - \E\big[ \tilde\omega-k_v(0) \, \big| \, A, \mathcal{F}_v\big] \E\big[ \tilde\omega-k_v(t) \, \big| \, A, \mathcal{F}_v\big] \P( A \, | \, \mathcal{F}_v),
\end{split}
\end{equation}
which hence gives
\begin{equation}\label{eq:inf_lb1}
\mathrm{Inf}_v(t)\ge\E\Big[\mathrm{Cov}\big(\tilde\omega-k_v(0),\tilde\omega-k_v(t) \,\big| \, A,\mathcal{F}_v  \big)\P(A|\mathcal{F}_v)\Big].
\end{equation}
Since $k_v(0)$ and $k_v(t)$ are $\mathcal{F}_v$-measureable, we have
\begin{equation}\label{eq:inf_lb2}
\mathrm{Cov}\big( \tilde\omega-k_v(0), \tilde\omega-k_v(t) \, \big| \, A, \mathcal{F}_v\big) =\mathrm{Cov}\big( \tilde\omega, \tilde\omega \, \big| \, A, \mathcal{F}_v\big) =
 \mathrm{Var}( \tilde\omega \, | \, A, \mathcal{F}_v).
\end{equation}
By assumption~\eqref{assumption2}, the above expression is uniformly bounded from below by a strictly positive constant $c>0$. Consequently, equations~\eqref{eq:inf_lb1} and~\eqref{eq:inf_lb2} together give that
$$
\mathrm{Inf}_v(t) \ge \E \big[\mathrm{Var}\left( \tilde\omega \, | \, A, \mathcal{F}_v\right) \P( A \, | \, \mathcal{F}_v) \big] \ge c \, \E \big[ \P( A \, | \, \mathcal{F}_v) \big].
$$
Finally, from~\eqref{eqprobevents2}, we obtain $\mathrm{Inf}_v(t) \ge c \, \P(v \in \pi_0 \cap \pi_t)$, as required.
\end{proof}

\section{From stability to chaos}
\label{sec:proofs}

Equipped with the covariance formula~\eqref{varianceinfluence}, and the connections between co-influences and geodesics derived in Section~\ref{sec:boundingtheinfluence}, we are now in a position to prove Theorem~\ref{thm:stabilityandchaos}.

\subsection{Proof of Theorem~\ref{thm:stabilityandchaos}, part~(i)}

The $L^2$-distance between $T_0$ and $T_t$ can be re-written into an expression similar to~\eqref{varianceinfluence}
\begin{equation}
\label{eq5}
\E\left[ (T_0 - T_t)^2 \right]  = 2 \, \E[T_0^2] - 2\, \E[T_0T_t]  = -2 \int_0^t \frac{d}{ds}Q_s \, ds = 2\int_0^t\sum_{v\in V}\mathrm{Inf}_v(s)\,ds.
\end{equation}
Since the co-influences are non-negative and non-increasing,~\eqref{eq5} and Lemma~\ref{lemmaUB} give that
\begin{equation}\label{eq:eq5}
\E\left[ (T_0 - T_t)^2 \right] \le 2t\sum_{v\in V}\mathrm{Inf}_v(0)\le 2t \sum_{v \in V} \E \big[ \omega_v(0)^2 \, \mathbbm{1}_{\{ v \in \pi_0 \}} \big] \le 2t\,\E\bigg[\sum_{v\in\pi_0}\omega_v(0)^2\bigg].
\end{equation}
The expectation in the right-hand side of~\eqref{eq:eq5} is bounded by the expected passage time for the last-passage problem where the weights have been replaced by the weights squared. As pointed out in Section 1.2, the condition \eqref{Fass} guarantees that the expected passage time in this setting grows linearly in $n$. Hence there exists $C<\infty$ such that
\begin{equation}
\label{eq6}
\E\left[ (T_0 - T_t)^2 \right] \leq 2C t n\quad\mbox{for all }n\geq 1.
\end{equation}

By expanding the square, we also note that
$$
\E\big[(T_0-T_t)^2\big]=\E[T_0^2]+\E[T_t^2]-2\E[T_0T_t]=\Var(T_0)+\Var(T_t)-2\Cov(T_0,T_t).
$$
Rearranging the terms of the above expression yields
\begin{equation}\label{eq:eq6}
\Cov(T_0,T_t)=\Var(T_0)-\frac12\E\big[(T_0-T_t)^2\big].
\end{equation}
Combining~\eqref{eq6} and~\eqref{eq:eq6} gives
$$
\Cov(T_0,T_t)\ge\Var(T_0)-Ctn.
$$
In particular, for $0<\alpha<\frac{n}{\Var(T)}$ and $t\le\alpha\frac{\Var(T)}{n}$, we conclude that
$$
\mathrm{Corr}(T_0,T_t)=\frac{\Cov(T_0,T_t)}{\Var(T_0)}\ge1-C\alpha,
$$
as required.

\subsection{Proof of Theorem~\ref{thm:stabilityandchaos}, part~(ii)}

The covariance formula~\eqref{varianceinfluence} together with Lemma~\ref{lemmaLB} gives the existence of a constant $c>0$ such that
\begin{equation}
\label{eq7}
\mathrm{Var}(T) \ge \int_0^t \sum_{v \in V} \mathrm{Inf}_v(s) \, ds \geq  c \int_0^t \sum_{v \in V} \P(v \in \pi_0 \cap \pi_s) \, ds  = c \int_0^t \E[|\pi_0 \cap \pi_s|] \, ds.
\end{equation}
Since $\P(v\in\pi_0\cap\pi_s)$ is non-increasing as a function of $s$ (recall the discussion related to~\eqref{eq:time_corr}), the same holds for $\E[|\pi_0\cap\pi_s|]$. Consequently, $\Var(T)\ge ct\,\E[|\pi_0\cap\pi_t|]$.
Hence, for $0<\alpha<\frac{n}{\Var(T)}$ and $t\ge\alpha\frac{\Var(T)}{n}$, we conclude that
$$
\E[|\pi_0 \cap \pi_t|] \leq \frac{1}{c t}\mathrm{Var}(T) \leq \frac{n}{c\alpha}.
$$
This completes the proof of Theorem~\ref{thm:stabilityandchaos}.

\appendix

\section{Appendix: Distributions satisfying \eqref{assumption2}}
\label{appA}

Here we demonstrate that power law distributions with finite variance and distributions with a (stretched) exponential tail decay satisfy the conditional variance assumption \eqref{assumption2}.

First consider a distribution with $\P(X > x) = Cx^{-\gamma}$ for $\gamma>2$ and $x \in [1, \infty)$, and note that
$$
\P(X > x \, | \, X > k) = \frac{\P(X > x)}{\P(X > k)}= \left(\frac{x}{k}\right)^{-\gamma}=C\,\P(kX>x)  \text{ for } x > k.
$$
Hence Var$(X|X>k)=\mbox{Var}(kX)=k^2\mbox{Var}(X)$, from which \eqref{assumption2} follows (in fact Var$(X|X>k)\to\infty$ with $k$).

Next consider a distribution with $\P(X > x) = Ce^{-\bar{\beta}x^{\beta}}$ for $\beta \in [0,1]$ and $x \in [0, \infty)$. To quantify $\E[X|X>k]$, note that
\begin{eqnarray*}
\E[X|X>k] & \sim & \sum_{n=0}^\infty\P(X>n|X>k)= k+e^{\bar{\beta}k^\beta}\sum_{n=k}e^{-\bar{\beta} n^\beta}\\
& \sim & k+Ce^{\bar{\beta}k^\beta}\int_{\bar{\beta}k^\beta}^\infty y^{\frac{1}{\beta}-1}e^{-y}dy,
\end{eqnarray*}
where the last equivalence follows from an integral approximation and the variable change $y=\bar{\beta}x^\beta$. Recalling the definition of the upper incomplete gamma function $\Gamma(s,x)=\int_x^\infty y^{s-1}e^{-y}dy$ and the fact that $\Gamma(s,x)\sim {x^{s-1}e^{-x}}$, we conclude that $\E[X|X>k]=k+b_k$, where $b_k\sim k^{1-\beta}$. Next note that it follows from Chebyshev's inequality that
$$
\mbox{Var}(X|X>k)\geq \P(X\leq k+b_k/2|X>k)\left(\frac{b_k}{2}\right)^2.
$$
We have that
$$
\P(X\leq k+b_k/2|X>k)=1-e^{-\bar{\beta}\left[(k+b_k/2)^\beta-k^\beta\right]},
$$
and hence
$$
\mbox{Var}(X|X>k)\geq \left(1-e^{-\bar{\beta}\left[(k+b_k/2)^\beta-k^\beta\right]}\right)\left(\frac{b_k}{2}\right)^2,
$$
which proves the claim. Indeed, we get a lower bound of constant order for $\beta=1$, while Var$(X|X>k)\to\infty$ as $k\to\infty$ for $\beta<1$.


\end{document}